\documentclass[a4paper,12pt]{amsart}
\usepackage{amsmath}
\usepackage{amssymb}
\usepackage{mathrsfs}
\usepackage{enumerate}
\usepackage{ifthen}
\usepackage{graphicx}
\usepackage{subfig}
\usepackage{caption}
\usepackage[T1]{fontenc} 
\usepackage{tikz}
\usepackage{cite}
\usepackage{mathtools}

\nonstopmode \numberwithin{equation}{section}
\newtheorem{theorem}{Theorem}[section]

\theoremstyle{remark}
\theoremstyle{definition}

\newtheorem{definition}{Definition}[section]

\theoremstyle{plain}

\newtheorem*{LemA}{Lemma A}
\newtheorem*{LemB}{Lemma B}
\newtheorem*{LemC}{Lemma C}
\newtheorem*{LemD}{Lemma D}
\newtheorem*{LemE}{Lemma E}
\newtheorem*{LemF}{Lemma F}

\numberwithin{equation}{section}
\numberwithin{theorem}{section}

\newcounter{minutes}
\divide\time by 60
\newcounter{hours}
\multiply\time by 60
\addtocounter{minutes}{-\time}

\begin{document}

\title{Results on Logarithmic Coefficients for the Class of Bounded Turning Functions}

\author{Sanju Mandal}
\address{Sanju Mandal, Department of Mathematics, Jadavpur University, Kolkata-700032, West Bengal, India.}
\email{sanjum.math.rs@jadavpuruniversity.in, 	sanju.math.rs@gmail.com}

\author{Molla Basir Ahamed$ ^* $}
\address{Molla Basir Ahamed, Department of Mathematics, Jadavpur University, Kolkata-700032, West Bengal, India.}
\email{mbahamed.math@jadavpuruniversity.in}


\def\thefootnote{}
\footnotetext{ {\tiny File:~\jobname.tex,
printed: \number\year-\number\month-\number\day,
          \thehours.\ifnum\theminutes<10{0}\fi\theminutes }
} \makeatletter\def\thefootnote{\@arabic\c@footnote}\makeatother

\begin{abstract}
It is crucial to explore the sharp bounds of logarithmic coefficients and the Hankel determinant involving logarithmic coefficients as part of coefficient problems in various function classes. Our primary objective in this study is to determine the sharp bounds for logarithmic coefficients as well as logarithmic inverse coefficients of bounded analytic functions associated with a bean-shaped domain in the class $\mathcal{BT_\mathfrak{B}}$. For this class, we also establish the sharp bounds for the second Hankel determinant involving logarithmic coefficients as well as logarithmic inverse coefficients. In addition, we establish sharp bounds for the generalized Zalcman conjecture inequality and the moduli differences of logarithmic coefficients for the class $\mathcal{BT_\mathfrak{B}}$.
\end{abstract}

\thanks{}
\maketitle
\pagestyle{myheadings}
\markboth{S. Mandal and M. B. Ahamed}{Logarithemic coefficient bounds of the Bounded Turning Functions}

\section{\bf Introduction}  
Let $\mathcal{H}$ denote the class of holomorphic functions $f$ in the open unit disk $\mathbb{D}=\{z\in\mathbb{C}: |z|<1\}$. Then $\mathcal{H}$ is a locally convex topological vector space endowed with the topology of uniform convergence over compact subsets of $\mathbb{D}$. Let $\mathcal{A}$ denote the class of functions $f\in\mathcal{H}$ such that $f(0)=0$ and $f^{\prime}(0)=1$ \textit{i.e.}, the function $f$ is of the form
\begin{align}\label{eq-1.1}
	f(z)=z+ \sum_{n=2}^{\infty}a_nz^n,\; \mbox{for}\; z\in\mathbb{D}.
\end{align} 
Let $\mathcal{S}$ denote the subclass of all functions in $\mathcal{A}$ which are univalent. For a comprehensive understanding of the theory of univalent functions and their significance in coefficient problems, we refer to the books \cite{Duren-1983-NY,Goodman-1983}.\vspace{2mm}

Before discussing some recent results and our main results of this paper, let us recall an important and useful tool known as the differential subordination technique. Many problems in geometric function theory can be solved effectively and precisely using this method.

\begin{definition}\label{def-1.1}
Let $f$ and $g$ be two analytic functions in the unit disk $\mathbb{D}$. Then $f$ is said to be subordinate to $g$, written as $f\prec g$ or $f(z)\prec g(z)$, if there exists a function $\omega$, analytic in $\mathbb{D}$ with $w(0)=0$, $|w(z)|<1$ such that $f(z)=g(w(z))$ for $z\in\mathbb{D}$. Moreover, if $g$ is univalent in $\mathbb{D}$ and $f(0)=g(0)$, then $f(\mathbb{D})\subseteq g(\mathbb{D})$.
\end{definition}

The most fundamental and significant subfamilies of the set $\mathcal{S}$ are the family $\mathcal{S}^*$ of starlike functions and the family $\mathcal{C}$ of convex functions which are defined as follows:
\begin{align*}
	\mathcal{S}^*=\left\{f\in\mathcal{A}: \frac{zf^{\prime}(z)} {f(z)}\prec\psi(z),\; z\in\mathbb{D}\right\}
\end{align*}
and
\begin{align*}
	\mathcal{C}=\left\{f\in\mathcal{A}: 1+ \frac{zf^{\prime\prime}(z)} {f^{\prime}(z)}\prec\psi(z),\; z\in\mathbb{D}\right\},
\end{align*}
with
\begin{align*}
	\psi(z)=1+2\sum_{n=2}^{\infty} z^n:=\frac{1+z}{1-z},\; z\in\mathbb{D}.
\end{align*}
A function $ f\in\mathcal{A} $ is called starlike (resp. convex) if the image $ f(\mathbb{D}) $ is a starlike domain
with respect to the origin (resp., convex). The classes of all starlike and convex functions that are univalent are denoted by $ \mathcal{S}^* $ and $ \mathcal{C} $, respectively. It is well-known that a function $ f $ in $ \mathcal{A} $ is starlike (resp. convex) if and only if $ {\rm Re}(zf^{\prime}(z)/f(z))>0 $ $ (\mbox{resp.}\; {\rm Re} (1+zf^{\prime\prime}(z)/f^{\prime}(z)))>0 $ for $ z\in\mathbb{D} $. By varying the function $\psi(z)$ in the above equations $\mathcal{S}^*$ and $\mathcal{C}$, we get some subfamilies which have significant geometric sense. The function $\mathfrak{B}(z)= \sqrt{1+\tanh z}$ represents a bean-shaped domain.\vspace{2mm}

Recently, Nandhini and Sruthakeerthi \cite{Nandhini-Sruthakeerthi-COAT-2025} defined a new subclass of bounded turning functions associated with a bean-shaped domain. There are several other subclasses that have been studied by researchers and each has significant geometrical properties. 
\begin{definition}\cite{Nandhini-Sruthakeerthi-COAT-2025}
Let $f\in\mathcal{A}$ is in the class $\mathcal{BT_\mathfrak{B}}$, if 
\begin{align*}
	\mathcal{BT_\mathfrak{B}}=\left\{f\in\mathcal{S}: f^{\prime}(z)\prec\mathfrak{B}(z)\right\}.
\end{align*}
\end{definition}
Note that 
\begin{align*}
	\mathfrak{B}(z)=\sqrt{1+\tanh z}=\sqrt{\frac{2}{(1+ e^{-2z})}}
\end{align*} and $\mathfrak{B}(z)$ conformally maps $\mathbb{D}$ onto
the region
\begin{align*}
	\Omega_{\mathfrak{B}}:=\left\{\omega\in\mathbb{C}:\;\vline\log\left(\frac{\omega^2}{2-\omega^2}\right)\vline<2\right\}.
\end{align*}
Geometrically, each $f\in\mathfrak{B}(z)$ maps to a bean-shaped region symmetric around the real axis, as shown in the following Fig. 1, left side is unit disk in $z$-plane and right side is $\omega$-plane.
\begin{figure}[!htb]
	\begin{center}
		\subfloat[Unit disk $z$-plane]{
		\includegraphics[width=0.36\linewidth]{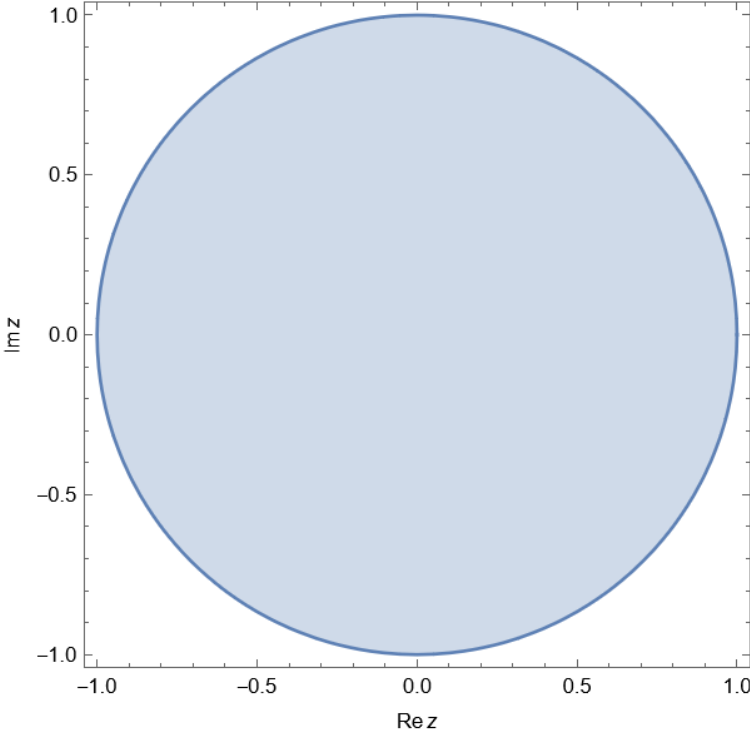}}
		\hspace{2mm}
		\subfloat[$w$-plane]{
		\includegraphics[width=0.52\linewidth]{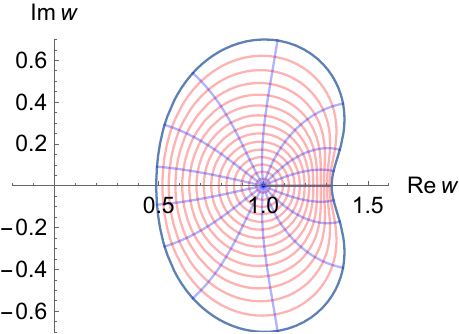}}
	\end{center}
	\caption{The image $ \mathfrak{B}(\mathbb{D}) $ is a bean shaped domain by the function $w=\mathfrak{B}(z)=\sqrt{1+\tanh z}$.}
\end{figure}\vspace{2mm}

Finding an upper bound for coefficients has been one of the central research topics in geometric function theory, as it reveals various properties of functions. The main challenge is to identify a suitable function from the class that effectively shown the sharpness of the bound. However, we point out that despite the extensive exploration of coefficient problems involving the Hankel determinant for the class $\mathcal{BT_\mathfrak{B}}$, the corresponding determinant with the logarithmic coefficient for the class $\mathcal{BT_\mathfrak{B}}$ has not garnered as much attention from researchers. Furthermore, the sharpness of the logarithmic coefficients has not been explored for the class $\mathcal{BT_\mathfrak{B}}$. This lack of attention serves as the primary motivation for the present paper and contribute to the understanding the several bounds of logarithmic coefficients for the class $\mathcal{BT_\mathfrak{B}}$. \vspace{2mm}

In this article, we aim to determine the sharp bounds for various problems in geometric function theory. These problems include finding the sharp bounds for logarithmic coefficients as well as logarithmic inverse coefficients and the sharp bound for the Hankel determinant of logarithmic coefficients as well as logarithmic inverse coefficients. In the subsequent sections, we will discuss our findings and provide a background study on these topics. The organization of this paper is as follows: In Section 3, we establish the sharp bounds for $\gamma_1, \gamma_2, \gamma_3$, and $\gamma_4$ and $|H_{2,1}(F_f/2)|$ for functions $f \in \mathcal{BT_\mathfrak{B}}$. In Section 4, we establish the sharp bound of $\Gamma_1$ and $\Gamma_2$ and $|H_{2,1}(F_{f^{-1}}/2)|$ for functions in the class $\mathcal{BT_\mathfrak{B}}$. In Section 5, we establish the sharp generalized Zalcman conjecture inequality for the class $\mathcal{BT_\mathfrak{B}}$. In Section 6, we establish the sharp moduli differences of logarithmic coefficients for the class $\mathcal{BT_\mathfrak{B}}$. The proofs of the results are discussed in detail in each respective section.

\section{\bf Sharp bound of logarithmic coefficients for functions in the class $\mathcal{A}$:}\label{Sec-2}
For $f\in\mathcal{S}$, we define the logarithmic coefficients $\gamma_{n}(f)$ by
\begin{align}\label{eq-2.1}
	F_{f}(z):=\log\dfrac{f(z)}{z}=2\sum_{n=1}^{\infty}\gamma_{n}(f)z^n, \;\; z\in\mathbb{D},\;\;\log 1:=0.
\end{align}
The logarithmic coefficients $\gamma_{n}$ for functions in the class $\mathcal{S}$ play a vital role in Milin’s conjecture (\cite{Milin-1977-ET}, see also \cite[p.155]{Duren-1983-NY}).
Milin conjectured that for $f\in\mathcal{S}$ and $n\geq2$,
\begin{align*}
	\sum_{m=1}^{n}\sum_{k=1}^{m}\left(k|\gamma_{k}|^2 -\frac{1}{k}\right)\leq 0,
\end{align*}
where the equality holds if, and only if, $f$ is a rotation of the Koebe function. De Branges \cite{Branges-AM-1985} has proved Milin conjecture which confirmed the famous Bieberbach conjecture. On the other hand, one of reasons for more attention has been given to the logarithmic coefficients is that the sharp bound for the class $\mathcal{S}$ is known only for $\gamma_{1}$ and $\gamma_{2}$, namely
\begin{align*}
	|\gamma_{1}|\leq 1, \;\; |\gamma_{2}|\leq \dfrac{1}{2}+ \dfrac{1}{e} =0.635\ldots
\end{align*}
It is still an open problem to find the sharp upper bounds for absolute value of $\gamma_{n}$, $n\geq 3$, for functions in the class $\mathcal{S}$. For the Koebe function \( k(z) = \frac{z}{(1-z)^2} \), the logarithmic coefficient is given by \( \gamma_n = \frac{1}{n} \). Since the Koebe function \( k \) serves as the extremal function for many extremal problems in the class \( \mathcal{S} \), it is anticipated that \( |\gamma_n| \leq \frac{1}{n} \) holds for functions in \( \mathcal{S} \). The problem of estimating the modulus of logarithmic coefficients for functions \( f \in \mathcal{S} \) with different settings and its various sub-classes has recently attracted the attention of several researchers. Recently, several researchers have shown interest in studying the logarithmic coefficients of functions in the class $ \mathcal{S} $ and its subclasses of $ \mathcal{S} $. For more information on logarithmic coefficients, we refer to \cite{Ali-Allu-PAMS-2018, Roth-PAMS-2007, Ali-Allu-Thomas-CRMCS-2018, Cho-Kowalczyk-kwon-Lecko-Sim-RACSAM-2020, Girela-AASF-2000,Thomas-PAMS-2016}.\vspace{2mm}

Establishing the sharp bounds of Hankel determinants of order 2 and 3 has been a major concern in geometric function theory, as it is directly related to coefficient problems. These determinants are formed by using the coefficients of analytic functions $ f $ which are represented by \eqref{eq-1.1} in the unit disk $\mathbb{D}$. Hankel matrices (and determinants) have emerged as fundamental elements in different areas of mathematics, finding a wide range of applications (see \cite{Ye-Lim-FCM-2016}). The primary objective of this study is to determine the sharp bound of logarithmic coefficients and the Hankel determinants involving the logarithmic coefficients. To start, we provide the definitions of Hankel determinants in the case that $f\in \mathcal{A}$.\vspace{2mm}

The Hankel determinant $H_{q,n}(f)$ of Taylor's coefficients of functions $f\in\mathcal{A}$ represented by \eqref{eq-1.1} is defined for $q,n\in\mathbb{N}$ as follows:
\begin{align*}
H_{q,n}(f):=\begin{vmatrix}
a_{n} & a_{n+1} &\cdots& a_{n+q-1}\\ a_{n+1} & a_{n+2} &\cdots& a_{n+q} \\ \vdots & \vdots & \vdots & \vdots \\ a_{n+q-1} & a_{n+q} &\cdots& a_{n+2(q-1)}
\end{vmatrix}.
\end{align*}
The extensive exploration of the sharp bounds of the Hankel determinants for starlike, convex, and other function classes have been undertaken in various studies (see \cite{ Kowalczyk-Lecko-RACSAM-2023, Raza-Riza-Thomas-BAMS-2023, Kowalczyk-Lecko-BAMS-2022, Mandal-Ahamed-LMJ-2024, Man-Roy-Aha-IJS-2024}), and their sharp bounds have been successfully established. \vspace{1.2mm}

Differentiating \eqref{eq-2.1} and then using \eqref{eq-1.1}, a simple computation shows that 
\begin{align}\label{eq-2.2}
	\begin{cases}
		\gamma_{1}=\dfrac{1}{2}a_{2},\vspace{2mm}\\ \gamma_{2}=\dfrac{1}{2} \left(a_{3} -\dfrac{1}{2}a^2_{2}\right), \vspace{2mm}\\ \gamma_{3} =\dfrac{1}{2}\left(a_{4}- a_{2}a_{3} +\dfrac{1}{3}a^3_{2}\right), \vspace{2mm}\\
		\gamma_{4}= \dfrac{1}{2} \left(a_{5} -a_{2}a_{4} +a^2_{2} a_{3} -\dfrac{1}{2}a^2_{3} -\dfrac{1}{4}a^4_{2}\right).
	\end{cases}
\end{align}
In $2022$, Kowalczyk and Lecko \cite{Kowalczyk-Lecko-BAMS-2022} proposed a Hankel determinant $H_{q,n}(F_f/2)$ whose elements are the logarithmic coefficients of $f\in\mathcal{S}$, realizing the extensive use of these coefficients. It follows that
\begin{align}\label{eq-2.3}
		H_{2,1}(F_f/2):=\gamma_1\gamma_3-\gamma_2^2=\frac{1}{48}\left(a_2^4-12a_3^2+12a_2a_4\right).
\end{align}
Furthermore, $H_{2,1}(F_{f}/2)$ is invariant under rotation, since for $f_{\theta}(z):=e^{-i\theta}f(e^{i\theta}z)$, $\theta\in\mathbb{R}$ when $f\in\mathcal{S}$, we have
\begin{align*}
	H_{2,1}(F_{f_{\theta}}/2)=\frac{e^{4i\theta}}{48}\left(a^4_2 - 12 a^2_3 + 12 a_2 a_4\right)=e^{4i\theta}H_{2,1}(F_{f}/2).
\end{align*}\vspace{2mm}

Let $\mathcal{P}$ be the class of all analytic functions $p$ in the unit disk $\mathbb{D}$ satisfying $p(0)=1$ and $\mbox{Re}\;p(z)>0$ for $z\in\mathbb{D}$. Therefore, every $p\in\mathcal{P}$ can be represented as
\begin{align}\label{eq-2.4}
	p(z)=1+\sum_{n=1}^{\infty}c_n z^n,\; z\in\mathbb{D}.
\end{align}
Elements of the class $\mathcal{P}$ are called  \textit{Carath$\acute{e}$odory functions}. It is well-known that $|c_n|\leq 2$, $n\geq 1$ for a functions $p\in\mathcal{P}$ (see \cite{Duren-1983-NY}). The Carath$\acute{e}$odory class $\mathcal{P}$ and it's coefficients bound play a significant role in establishing the bound of Hankel determinants.\vspace{2mm}

Now, we state some lemmas, which will be useful to establish our main results. Parametric representations of the coefficients are often useful in finding the bound for Hankel determinants, and in this regard, Libera and Zlotkiewicz (see \cite{Libera-Zlotkiewicz-PAMS-1982, Libera-Zlotkiewicz-PAMS-1983}) obtained the parameterizations of possible values of $c_2$ and $c_3$, which are Taylor coefficients for functions with positive real part.\vspace{2mm}

\begin{LemA}\cite{Libera-Zlotkiewicz-PAMS-1982,Libera-Zlotkiewicz-PAMS-1983}
If $p\in\mathcal{P}$ is of the form \eqref{eq-2.4} with $c_1\geq 0$, then 
\begin{align}\label{eq-2.5}
	&c_1=2\tau_1,\\\label{eq-2.6} &c_2=2\tau^2_1 +2(1-\tau^2_1)\tau_2
\end{align}
and
\begin{align}\label{eq-2.7}
	c_3= 2\tau^3_1  + 4(1 -\tau^2_1)\tau_1\tau_2 - 2(1 - \tau^2_1)\tau_1\tau^2_2 + 2(1 - \tau^2_1)(1 - |\tau_2|^2)\tau_3
\end{align}
for some $\tau_1\in[0,1]$ and $\tau_2,\tau_3\in\overline{\mathbb{D}}:= \{z\in\mathbb{C}:|z|\leq 1\}$.\vspace{1.2mm}
	
For $\tau_1\in\mathbb{T}:=\{z\in\mathbb{C}:|z|=1\}$, there is a unique function $p\in\mathcal{P}$ with $c_1$ as in \eqref{eq-2.5}, namely
\begin{align*}
	p(z)=\frac{1+\tau_1 z}{1-\tau_1 z}, \;\;z\in\mathbb{D}.
\end{align*}
	
For $\tau_1\in\mathbb{D}$ and $\tau_2\in\mathbb{T}$, there is a unique function $p\in\mathcal{P}$ with $c_1$ and $c_2$ as in \eqref{eq-2.5} and \eqref{eq-2.6}, namely
\begin{align*}
	p(z)=\frac{1+(\overline{\tau_1}\tau_2 +\tau_1)z+\tau_2 z^2}{1 +(\overline{\tau_1}\tau_2 -\tau_1)z-\tau_2 z^2}, \;\;z\in\mathbb{D}.
\end{align*}
	
For $\tau_1,\tau_2\in\mathbb{D}$ and $\tau_3\in\mathbb{T}$, there is a unique function $p\in\mathcal{P}$ with $c_1,c_2$ and $c_3$ as in \eqref{eq-2.5}--\eqref{eq-2.7}, namely
\begin{align*}
	p(z)=\frac{1+(\overline{\tau_2}\tau_3+\overline{\tau_1}\tau_2 +\tau_1)z +(\overline{\tau_1}\tau_3+ \tau_1\overline{\tau_2}\tau_3 +\tau_2)z^2 +\tau_3 z^3}{1 +(\overline{\tau_2}\tau_3+ \overline{\tau_1}\tau_2 -\tau_1)z +(\overline{\tau_1}\tau_3- \tau_1\overline{\tau_2}\tau_3 -\tau_2)z^2 -\tau_3 z^3}, \;\;z\in\mathbb{D}.
\end{align*}
\end{LemA}
The next lemma is a special case of more general results due to Choi \textit{et al.} \cite{Cho-Kim-Sugawa-JMSJ-2007} (see also \cite{Ohno-Sugawa-KJM-2018}). We will also apply the following lemma. Here $\overline{\mathbb{D}}:= \{z\in\mathbb{C}:|z|\leq 1\}$. A more general and symmetric problem was considered in \cite{Cho-Kim-Sugawa-JMSJ-2007}. Let 
\begin{align*}
	\Omega(A,B,K,L,M)=\max_{u,v\in\overline{\mathbb{D}}}\{|A|(1-|u|^2) +|B|(1-|v|^2)+|Ku^2 +2Muv+Lv^2|\}
\end{align*}
for $A, B, K, L, M\in\mathbb{C}$. When $K, L, M$ are real numbers, the value of $\Omega(A,B,K,L,M)$ was compute in \cite[Theorem 3.1]{Cho-Kim-Sugawa-JMSJ-2007}. By virtue of the maximum modulus principle, one can see that
\begin{align*}
	\Omega(1,0,c,a,b/2)=\max_{u\in\overline{\mathbb{D}}, v\in\partial\mathbb{D}}\{(1-|u|^2) +|cu^2 +buv+av^2|\}=Y(a,b,c).
\end{align*}
As an immediate consequence of \cite[Theorem 3.1]{Cho-Kim-Sugawa-JMSJ-2007}, the following result was obtained.
\begin{LemB}\cite[Proposition 6]{Cho-Kim-Sugawa-JMSJ-2007}
Let $A, B, C$ be real numbers and
\begin{align*}
	Y(A,B,C):=\max\{|A+ Bz +Cz^2| +1-|z|^2: z\in\overline{\mathbb{D}}\}.
\end{align*}
\noindent{(i)} If $AC\geq 0$, then
\begin{align*}
	Y(A,B,C)=\begin{cases}
		|A|+|B|+|C|, \;\;\;\;\;\;\;\;\;\;\;\;\;|B|\geq 2(1-|C|), \vspace{2mm}\\ 1+|A|+\dfrac{B^2}{4(1-|C|)}, \;\;\;\;\;|B|< 2(1-|C|).
	\end{cases}
\end{align*}
\noindent{(ii)} If $AC<0$, then
\begin{align*}
	Y(A,B,C)=\begin{cases}
		1-|A|+\dfrac{B^2}{4(1-|C|)}, \;\;\;\;-4AC(C^{-2}-1)\leq B^2\land|B|< 2(1-|C|),\vspace{2mm} \\ 1+|A|+\dfrac{B^2}{4(1+|C|)}, \;\;\;\; B^2<\min\{4(1+|C|)^2,-4AC(C^{-2}-1)\}, \vspace{2mm} \\ R(A,B,C), \;\;\;\;\;\;\;\;\;\;\;\;\;\;\;\;\;\;\; otherwise,
	\end{cases}
\end{align*}
where
\begin{align*}
	R(A,B,C):= \begin{cases}
		|A|+|B|-|C|, \;\;\;\;\;\;\;\;\;\;\;\;\; |C|(|B|+4|A|)\leq |AB|, \vspace{2mm}\\ -|A|+|B|+|C|, \;\;\;\;\;\;\;\;\;\;\; |AB|\leq |C|(|B|-4|A|), \vspace{2mm}\\ (|C| +|A|)\sqrt{1-\dfrac{B^2}{4AC}}, \;\;\; otherwise.
	\end{cases}
\end{align*}
\end{LemB}

\begin{LemC}\cite{Ma-Minda-1994}
Let $p\in\mathcal{P}$ be given by \eqref{eq-2.1}. Then
\begin{align*}
	|c_2 -vc^2_1|\leq\begin{cases}
		-4v +2 \;\;\;\; v<0,\\ 2 \;\;\;\;\;\;\;\;\;\;\;\;\;\; 0\leq v\leq 1, \\ 4v -2 \;\;\;\;\;\;\; v>1.
	\end{cases}
\end{align*}
For $v<0$ or $v>1$, the equality holds if, and only if,
\begin{align*}
	h(z)=\frac{1+z}{1-z}
\end{align*}
or one of its rotations. If $0 < v < 1$, then the equality is true if, and only if,
\begin{align*}
	h(z)=\frac{1+z^2}{1-z^2}
\end{align*}
or one of its rotations.
\end{LemC}

\begin{LemD}\cite{Ali-BMMSS-2001}
Let $p\in\mathcal{P}$ be given by \eqref{eq-2.1} with $0\leq B\leq 1$ and $B(2B-1)\leq D\leq B$. Then 
\begin{align*}
	|c_3 -2Bc_1 c_2 +Dc^3_1|\leq 2.
\end{align*}
\end{LemD}
\begin{LemE}\cite{Ravichandran-Verma-CRMAS-2015}
Let $p\in\mathcal{P}$ be given by \eqref{eq-2.4}. If $\alpha,\beta, \gamma$ and $\lambda$ satisfying the conditions $0<\alpha<1$, $0 <\lambda<1$ and
\begin{align*}
	8\lambda(1-\lambda)\{(\alpha\beta -2\gamma)^2 &+(\alpha(\lambda+\alpha)-\beta)^2\} \\&+\alpha(1-\alpha)(\beta -2\lambda \alpha)^2 \leq 4 \alpha^2(1-\alpha)^2\lambda(1-\lambda),
\end{align*}
then 
\begin{align*}
	\vline\;\gamma c^4_1 +\lambda c^2_2 +2\alpha c_1 c_3 -\frac{3}{2}\beta c^2_1 c_2 -c_4\;\vline\leq 2.
\end{align*}
\end{LemE}

\section{\bf Logarithmic coefficients for the Class $\mathcal{BT_\mathfrak{B}}$:} The significance of logarithmic coefficients in geometric function theory has led to a growing interest in finding sharp bound of logarithmic coefficients and the Hankel determinants with these coefficients. We obtain the following sharp bound of logarithmic coefficients for the class $\mathcal{BT_\mathfrak{B}}$.

\begin{theorem}
Let $f(z)=z+a_2z^2+a_3z^3+\cdots\in\mathcal{BT_\mathfrak{B}}$ and $ \gamma_{1}, \gamma_{2}, \gamma_{3}, \gamma_{4}$ are given by \eqref{eq-2.2}. Then we have 
\begin{align*}
	|\gamma_n|\leq \frac{1}{4(n+1)}\;\;\;\;\mbox{for}\;\;n=1,2,3,4.
\end{align*}
The inequality is sharp for the following functions:
\begin{align}\label{eq-2.8}
	f_1(z)=\int_{0}^{z}\left(\sqrt{1+\tanh t}\right) dt =z +\frac{z^2}{4} -\frac{z^3}{24} -\frac{5z^4}{192} +\frac{17z^5}{1920}+\cdots.
\end{align}
\begin{figure}[!htb]
	\begin{center}
		\subfloat[Unit disk $z$-plane]{
			\includegraphics[width=0.42\linewidth]{Graph-01}}
		\hspace{1cm}
		\subfloat[$w$-plane]{
			\includegraphics[width=0.48\linewidth]{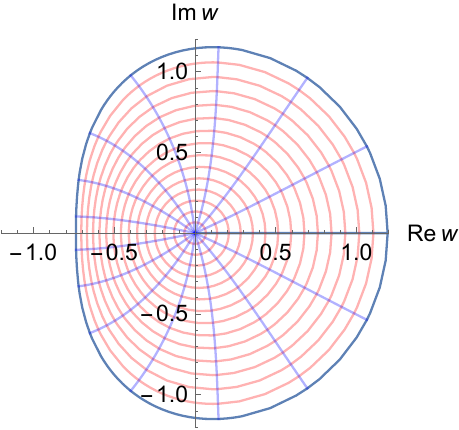}}
	\end{center}
	\caption{Pictorial representation of $f_1(z)=z +\frac{z^2}{4} -\frac{z^3}{24} -\frac{5z^4}{192} +\cdots$.}
\end{figure}
\begin{align}\label{eq-2.9}
	f_2(z)=\int_{0}^{z}\left(\sqrt{1+\tanh t^2}\right) dt =z +\frac{z^3}{6} -\frac{z^5}{40}-\frac{5z^7}{336}+\cdots.
\end{align}
\begin{figure}[!htb]
	\begin{center}
		\subfloat[Unit disk $z$-plane]{
			\includegraphics[width=0.42\linewidth]{Graph-01}}
		\hspace{1cm}
		\subfloat[$w$-plane]{
			\includegraphics[width=0.48\linewidth]{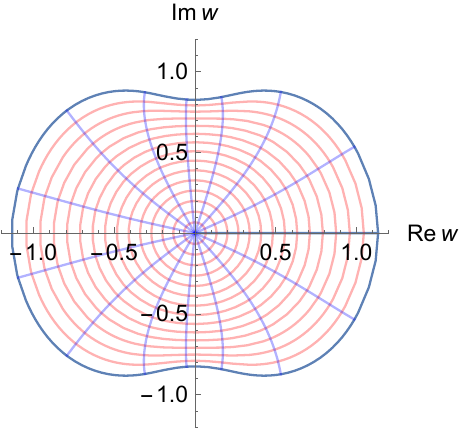}}
	\end{center}
	\caption{Pictorial representation of $f_2(z)=z +\frac{z^3}{6} -\frac{z^5}{40}-\frac{5z^7}{336}+\cdots$.}
\end{figure}
\begin{align}\label{eq-2.10}
	f_3(z)=\int_{0}^{z}\left(\sqrt{1+\tanh t^3}\right) dt =z +\frac{z^4}{8} -\frac{z^7}{56}+\cdots.
\end{align}
\begin{figure}[!htb]
	\begin{center}
		\subfloat[Unit disk $z$-plane]{
			\includegraphics[width=0.42\linewidth]{Graph-01}}
		\hspace{1cm}
		\subfloat[$w$-plane]{
			\includegraphics[width=0.48\linewidth]{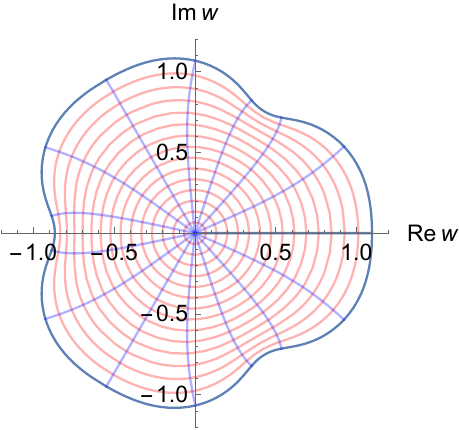}}
	\end{center}
	\caption{Pictorial representation of $f_3(z)=z +\frac{z^4}{8} -\frac{z^7}{56}+\cdots$.}
\end{figure}
\begin{align}\label{eq-2.11}
	f_4(z)=\int_{0}^{z}\left(\sqrt{1+\tanh t^4}\right) dt =z +\frac{z^5}{10} -\frac{z^9}{72}+\cdots.
\end{align}
\begin{figure}[!htb]
	\begin{center}
		\subfloat[Unit disk $z$-plane]{
			\includegraphics[width=0.42\linewidth]{Graph-01}}
		\hspace{1cm}
		\subfloat[$w$-plane]{
			\includegraphics[width=0.48\linewidth]{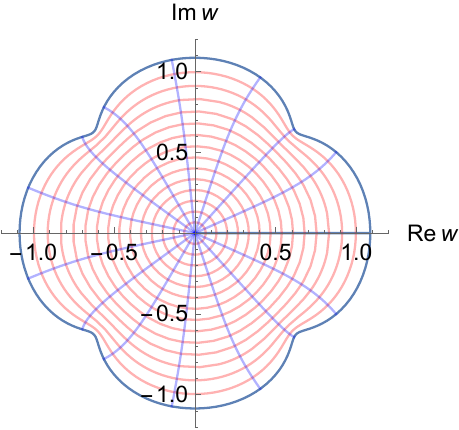}}
	\end{center}
	\caption{Pictorial representation of $f_4(z)=z +\frac{z^5}{10} -\frac{z^9}{72}+\cdots$.}
\end{figure}
\end{theorem}
\begin{proof}
Let $f\in\mathcal{BT_\mathfrak{B}}$. Then there exists an analytic function $\omega\in\mathcal{H}$ with $\omega(0)=0$ and $|\omega(z)|<1$ for $z\in\mathbb{D}$, such that 
\begin{align}\label{eq-2.12}
	f^{\prime}(z)=\mathfrak{B}(\omega(z))=\sqrt{1+\tanh(\omega(z))},\;\; z\in\mathbb{D}.
\end{align}
Let $p\in\mathcal{P}$. Then, it can be represented using the Schwarz function as $\omega$ by
\begin{align*}
	p(z)=\frac{1+\omega(z)}{1-\omega(z)}=1+c_1z+c_2z^2+\cdots,\; z\in\mathbb{D}.
\end{align*} 
Hence, it is evident that
\begin{align}\label{eq-2.13}
	\omega(z)&=\nonumber\frac{p(z)-1}{p(z)+1}\\&\nonumber=\frac{c_1}{2}z+\frac{1}{2}\left(c_2-\frac{c_1^2}{2}\right)z^2+\frac{1}{2}\left(c_3-c_1c_2+\frac{c_1^3}{4}\right)z^3\\&\quad+\frac{1}{2}\left(c_4-c_1c_3+\frac{3c_1^2c_2}{4}-\frac{c_2^2}{2}-\frac{c_1^4}{8}\right)z^4+\cdots
\end{align}
in $\mathbb{D}$. Then $p$ is analytic in $\mathbb{D}$ with $p(0)=1$ and has positive real part in $\mathbb{D}$. In view of \eqref{eq-2.12} together with $\mathcal{B}(\omega(z))$, a tedious computation shows that 
\begin{align}\label{eq-2.14}
	\nonumber\sqrt{1+\tanh(\omega(z))}&=1+\frac{1}{4}c_1z+\left( \frac{1}{4}c_2 -\frac{5}{32}c_1^2\right)z^2+ \left(\frac{1}{4} c_3 -\frac{5}{16}c_1c_2 +\frac{31}{384}c^3_1\right)z^3\\&\quad + \left(\frac{1}{4}c_4 -\frac{1}{4}c_1c_3 -\frac{5}{32}c_2^2 + \frac{7}{32}c^2_1 c_2 -\frac{223}{6144}c_1^4\right)z^4+\cdots
\end{align}
and 
\begin{align}\label{eq-2.15}
	f^{\prime}(z)=1+2a_2z+3a_3 z^2 +4a_4 z^3+ 5a_5 z^4+\cdots.
\end{align}
Thus, applying \eqref{eq-2.14} and \eqref{eq-2.15}, we derive from \eqref{eq-2.12} that
\begin{align}\label{eq-2.16}
	\begin{cases}
		a_2=\dfrac{1}{8}c_1,\vspace{2mm}\\
		a_3=\dfrac{1}{12}c_2 - \dfrac{5}{96}c_1^2,\vspace{2mm}\\
		a_4=\dfrac{1}{16}c_3 - \dfrac{5}{64}c_1c_2 + \dfrac{31}{1536}c_1^3,\vspace{2mm}\\ 
		a_5=\dfrac{1}{20}c_4 -\dfrac{1}{20}c_1c_3 -\dfrac{1}{32}c_2^2 + \dfrac{7}{160}c^2_1 c_2 -\dfrac{223}{30720}c_1^4. 
	\end{cases}
\end{align}\vspace{2mm}

\noindent{\bf Sharp bounds of $\gamma_{1}$:} By using \eqref{eq-2.2} and \eqref{eq-2.16}, we obtain
\begin{align}\label{eq-2.17}
	|\gamma_{1}|=\;\vline\frac{1}{2}a_{2}\;\vline=\frac{1}{16}|c_{1}|\leq \frac{1}{8}.
\end{align}
The desired bound is obtained. The function $f_1$, which is defined in \eqref{eq-2.8} gives the sharpness of the inequality \eqref{eq-2.17}. \vspace{2mm}
	
\noindent{\bf Sharp bounds of $\gamma_{2}$:} Using \eqref{eq-2.2} and \eqref{eq-2.16}, we observe that
\begin{align*}
	|\gamma_{2}|&=\;\vline\frac{1}{2} \left(a_{3} -\dfrac{1}{2}a^2_{2} \right)\vline \vspace{2mm} \\&= \;\vline\frac{1}{2}\left( \dfrac{1}{12}c_2 - \dfrac{5}{96}c_1^2 -\frac{1}{2} \left(\frac{1}{8}c_1\right)^2\right)\vline \vspace{2mm} \\&= \frac{1}{24}\;\vline\; c_2-\frac{23}{32}c^2_{1} \;\vline. 
\end{align*}
Hence, from Lemma C, we derive
\begin{align}\label{eq-2.18}
	|\gamma_{2}|\leq\frac{1}{12}.
\end{align}
The function $f_2$, which is defined in \eqref{eq-2.9} gives the sharpness of the inequality \eqref{eq-2.18}.\vspace{2mm}

\noindent{\bf Sharp bounds of $\gamma_{3}$:} By using \eqref{eq-2.2} and \eqref{eq-2.12}, we obtain
\begin{align*}
	|\gamma_{3}| &=\;\vline\frac{1}{2}\left(a_{4}- a_{2}a_{3} +\frac{1}{3}a^3_{2}\right)\vline \vspace{2mm} \\&=\; \vline\frac{1}{1536} \left(21c_1^3- 68 c_1c_2+48c_3\right)\vline \vspace{2mm} \\&=\frac{1}{32}\;\vline\; c_3 -2Bc_1 c_2 +Dc^3_1\;\vline,
\end{align*}
where $B=\frac{17}{24}$ and $D=\frac{7}{16}$.\vspace{2mm}

 Therefore, it is clear that $0\leq B\leq 1$, and the inequality $B(2B-1)\leq D\leq B$ holds. In view of the Lemma D, we obtain
\begin{align}\label{eq-2.19}
	|\gamma_{3}|\leq\frac{1}{16}.
\end{align}
The function $f_3$, defined in \eqref{eq-2.10}, establishes the sharpness of the inequality \eqref{eq-2.19}.\vspace{2mm}

\noindent{\bf Sharp bounds of $\gamma_{4}$:} By employing \eqref{eq-2.2} and \eqref{eq-2.12}, we find that
\begin{align*}
	|\gamma_{4}| &=\;\vline\frac{1}{2}\left(a_{5}- a_{2}a_{4} +a^2_2 a_3  -\frac{1}{2}a^2_{3} -\frac{1}{4}a^4_2\right)\vline \vspace{2mm} \\&=\; \vline\frac{1}{1474560} \left(-8857 c^4_1 +43616c^2_1 c_2 -25600c^2_2 -42624c_1 c_3 +36864 c_4\right)\vline \vspace{2mm} \\&=\frac{1}{40}\;\vline\; \frac{8857}{36864} c^4_1 +\frac{25}{36}c^2_2 +\frac{37}{32}c_1 c_3 -\frac{1363}{1152} c^2_1 c_2 -c_4 \;\vline \\&=\frac{1}{40}\;\vline\;\gamma c^4_1 +\lambda c^2_2 +2\alpha c_1 c_3 -\frac{3}{2}\beta c^2_1 c_2 -c_4\;\vline,
\end{align*}
where
\begin{align*}
	\gamma=\frac{8857}{36864}\;\;\;\; \lambda=\frac{25}{36} \;\;\;\; \alpha=\frac{37}{64} \;\;\;\;\beta=\frac{1363}{1728}.
\end{align*}
Now, a thorough calculation establishes that
\begin{align*}
	8\lambda(1-\lambda)&\{(\alpha\beta -2\gamma)^2 +(\alpha(\lambda+\alpha)-\beta)^2\} \\&+\alpha(1-\alpha)(\beta -2\lambda \alpha)^2 - 4 \alpha^2(1-\alpha)^2\lambda(1-\lambda)= -\frac{22111611107}{495338913792}<0.
\end{align*}
As a result, using Lemma E, we conclude
\begin{align}\label{eq-2.20}
	|\gamma_{4}|\leq\frac{1}{20}.
\end{align}
The function $f_4$, defined in \eqref{eq-2.11}, confirms the sharpness of the inequality \eqref{eq-2.20}. This completes the proof.
\end{proof}\vspace{2mm}

\subsection{\bf Sharp bound of $H_{2,1}(F_f/2)$ for the Class $\mathcal{BT_\mathfrak{B}}$:} We obtain a result finding the sharp bound of the second Hankel determinant $ H_{2,1}(F_f/2) $ with logarithmic coefficients for functions in the class $ \mathcal{BT_\mathfrak{B}}$.
\begin{theorem}\label{Th-2.1}
Let $f\in\mathcal{BT_\mathfrak{B}}$ and has the series representation $f(z)=z+a_2z^2+a_3z^3+\cdots$, and $ \gamma_{1}, \gamma_{2} $, $ \gamma_{3} $ are given by \eqref{eq-2.2}. Then we have 
\begin{align*}
	|H_{2,1}(F_f/2)|:=|\gamma_1\gamma_3-\gamma_2^2|\leq \frac{1}{144}.
\end{align*}
The inequality is sharp for the function $f_2$, which is defined in \eqref{eq-2.9}.
\end{theorem}
\begin{proof}
Let $f\in\mathcal{BT_\mathfrak{B}}$. Then there exists an analytic function $\omega\in\mathcal{H}$ with $\omega(0)=0$ and $|\omega(z)|<1$ for $z\in\mathbb{D}$, such that 
\begin{align}\label{eq-2.21}
	f^{\prime}(z)=\mathfrak{B}(\omega(z))=\sqrt{1+\tanh(\omega(z))},\;\; z\in\mathbb{D}.
\end{align}
Let $p\in\mathcal{P}$. Then, it may be written in terms of the Schwarz function $\omega$ by
\begin{align*}
	p(z)=\frac{1+\omega(z)}{1-\omega(z)}=1+c_1z+c_2z^2+\cdots,\; z\in\mathbb{D}.
\end{align*} 
Therefore, it is clear that
\begin{align}\label{eq-2.22}
	\omega(z)&=\nonumber\frac{p(z)-1}{p(z)+1}\\&\nonumber=\frac{c_1}{2}z+\frac{1}{2}\left(c_2-\frac{c_1^2}{2}\right)z^2+\frac{1}{2}\left(c_3-c_1c_2+\frac{c_1^3}{4}\right)z^3\\&\quad+\frac{1}{2}\left(c_4-c_1c_3+\frac{3c_1^2c_2}{4}-\frac{c_2^2}{2}-\frac{c_1^4}{8}\right)z^4+\cdots
\end{align}
in $\mathbb{D}$. Then $p$ is analytic in $\mathbb{D}$ with $p(0)=1$ and has positive real part in $\mathbb{D}$. In view of \eqref{eq-2.21} together with $\mathcal{B}(\omega(z))$, a tedious computation shows that 
\begin{align}\label{eq-2.23}
	\nonumber\sqrt{1+\tanh(\omega(z))}&=1+\frac{1}{4}c_1z+\left( \frac{1}{4}c_2 -\frac{5}{32}c_1^2\right)z^2+ \left(\frac{1}{4} c_3 -\frac{5}{16}c_1c_2 +\frac{31}{384}c^3_1\right)z^3\\&\quad + \left(\frac{1}{4}c_4 -\frac{1}{4}c_1c_3 -\frac{5}{32}c_2^2 + \frac{7}{32}c^2_1 c_2 -\frac{223}{6144}c_1^4\right)z^4+\cdots
\end{align}
and 
\begin{align}\label{eq-2.24}
	f^{\prime}(z)=1+2a_2z+3a_3 z^2 +4a_4 z^3+ 5a_5 z^4+\cdots.
\end{align}
Thus, using \eqref{eq-2.23} and \eqref{eq-2.24}, we compute from \eqref{eq-2.12} that
\begin{align}\label{eq-2.25}
	\begin{cases}
		a_2=\dfrac{1}{8}c_1,\vspace{2mm}\\
		a_3=\dfrac{1}{12}c_2 - \dfrac{5}{96}c_1^2,\vspace{2mm}\\
		a_4=\dfrac{1}{16}c_3 - \dfrac{5}{64}c_1c_2 + \dfrac{31}{1536}c_1^3. 
	\end{cases}
\end{align}	
A simple computation by using \eqref{eq-2.3} and \eqref{eq-2.25}, shows
that
\begin{align}\label{eq-2.26}
	H_{2,1}(F_f/2)&\nonumber=\frac{1}{48}\left(a_2^4-12a_3^2+12a_2a_4\right)\\&=\frac{1}{589824}\left(-25c_1^4 - 160c_1^2c_2 - 1024c_2^2 + 1152c_1c_3\right).
\end{align}
By Lemma A and \eqref{eq-2.26}, we obtain
\begin{align}\label{eq-2.27}
	H_{2,1}(F_f/2)\nonumber&=\frac{1}{36864} \bigg(-73\tau_1^4 -16\tau_1^2\tau_2\left(1-\tau_1^2\right)-32\tau_2^2(8+\tau_1^2)(1-\tau_1^2)\\&\quad+ 288\tau_1\tau_3 \left(1-\tau_1^2\right) \left(1-|\tau_2|^2\right)\bigg).
\end{align}
	 
 We now explore three possible cases involving $\tau_1$. \vspace{1.2mm}
	 
\noindent{\bf Case-I.} Let $\tau_1=1$. Then, from \eqref{eq-2.27} we see that 
\begin{align*}
	|H_{2,1}(F_f/2)|=\frac{73}{36864}\approx 0.00198025.
\end{align*}
\noindent{\bf Case-II.} Let $\tau_1=0$. Then, from \eqref{eq-2.27} we get 
\begin{align*}  
	|H_{2,1}(F_f/2)|=\bigg|\frac{1}{36864}\left(-256\tau_2^2\right)\bigg|\leq\frac{1}{144}\approx 0.0069444.
\end{align*}
\noindent{\bf Case-III.} Let $\tau_1\in (0, 1)$. Applying triangle inequality in \eqref{eq-2.27} and using the fact that $|\tau_3|\leq 1$, we obtain
\begin{align}\label{eq-2.28}
	|H_{2,1}(F_f/2)|&\nonumber\leq\frac{1}{36864}\bigg(\bigg|-73\tau_1^4 -16\tau_1^2\tau_2\left(1-\tau_1^2\right)-32\tau_2^2(8+\tau_1^2) (1-\tau_1^2)\bigg|\\&\nonumber\quad+288\tau_1(1-\tau_1^2)(1-|\tau_2|^2)\bigg)\\&=\frac{1}{128}\tau_1(1-\tau_1^2)\left(\vline\; A+B\tau_2+C\tau_2^2\;\vline +1-|\tau_2|^2\right)\nonumber\\& :=\frac{1}{128}\tau_1(1-\tau_1^2)Y(A, B, C),
\end{align}
where
\begin{align*}
	A=\frac{73\tau_1^3}{288(1-\tau_1^2)},\;\; B=\frac{\tau_1}{18},\;\;\mbox{and}\;\; C=\frac{(8+\tau_1^2)}{9\tau_1}.
\end{align*}
We note that $AC>0$. Hence, we can apply case (i) of Lemma B and discuss the following cases.\\
			
A simple computation shows that
\begin{align*}
	|B|-2(1-|C|)=\frac{\tau_1}{18}-2\left(1-\frac{8+\tau_1^2}{9\tau_1}\right)=\frac{32-36\tau_1+5\tau_1^2}{18\tau_1}>0
\end{align*}
for all $\tau_1\in (0, 1)$. \textit{i.e.,} $|B|>2(1-|C|)$. Thus from Lemma B, we see that 
\begin{align*}
	Y(A, B, C)=|A|+|B|+|C|=\frac{256- 208\tau_1^2 + 25\tau_1^4} {288\tau_1\left(1-\tau_1^2\right)}.
\end{align*}
In view of the inequality \eqref{eq-2.28} it follows that
\begin{align}\label{eq-2.29}
	\nonumber|H_{2,1}(F_f/2)|&=\frac{1}{128}\tau_1(1-\tau_1^2)\left(|A|+|B|+|C|\right)\\&\nonumber=\frac{1}{36864}\left(256- 208\tau_1^2 + 25\tau_1^4\right)\\&=\frac{1}{36864}\Psi(\tau_1),
\end{align}
\begin{figure}[!htb]
	\begin{center}
		\includegraphics[width=0.55\linewidth]{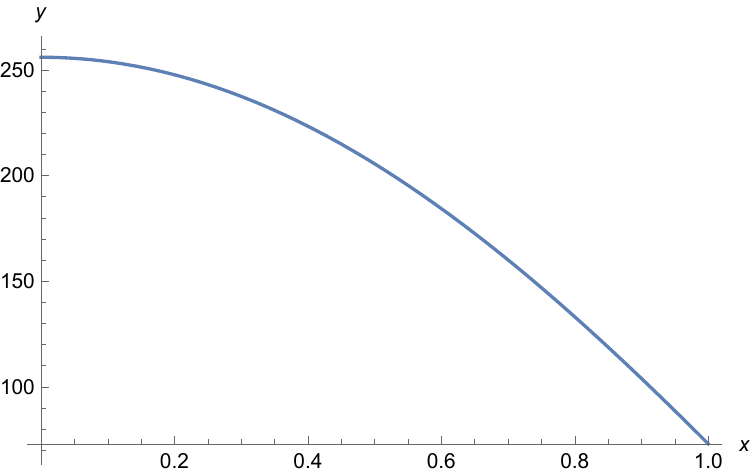}
	\end{center}
	\caption{The image $\Psi(t)$ for $0\leq t\leq 1$.}
\end{figure}
where $\Psi(t)=256-208t^2+25t^4$ for $t\in [0, 1]$. A simple computation shows that $\Psi^{\prime}(t)=-416t + 100t^3<0$ for all $t\in (0,1]$ which shows that $\Psi$ is a decreasing function on $[0, 1]$. Hence, the maximum of $\Psi(t)$ is attained at $t=0$, and the maximum value is $256$. Hence, from \eqref{eq-2.29}, we see that
\begin{align}\label{eq-2.30}
	|H_{2,1}(F_{f/2})|\leq \frac{1}{144}.
\end{align}
By summarizing Cases I, II, and III, we obtain the desired inequality of the result. The function $f_2$, which is defined in \eqref{eq-2.9} gives the sharpness of the inequality \eqref{eq-2.30}. This completes the proof.
\end{proof}

\section{\bf Sharp bound of logarithmic coefficients for inverse functions in the class $\mathcal{BT_\mathfrak{B}}$}
Let $F$ be the inverse function of $f\in\mathcal{S}$ defined in a neighborhood of the origin with the Taylor series expansion
\begin{align}\label{eq-3.1}
	F(w):=f^{-1}(w)= w+\sum_{n=2}^{\infty} A_n w^n,
\end{align}
where we may choose $|w|<1/4$, as we know that the famous Koebe’s $1/4$-theorem ensures that, for each univalent function $f$ defined in $\mathbb{D}$, it inverse $f^{-1}$ exists at least on a disc of radius $1/4$. Using a variational method, L$\ddot{\mbox{o}}$wner \cite{Lowner-IMA-1923} has obtained the sharp estimate $ |A_n|\leq K_n \;\mbox{for each}\; n, $ where $K_n=(2n)!/(n!(n+1)!)$ and $K(w)= w +K_2 w^2 +K_3 w^3 +\cdots$ is the inverse of the K\"oebe function. There has been a good deal of interest in determining the behavior of the inverse coefficients of $f$ given in \eqref{eq-1.1} when the corresponding function $f$ is restricted to some proper geometric subclasses of $\mathcal{S}$.\vspace{2mm}

Let $f(z)=z+ \sum_{n=2}^{\infty}a_nz^n$ be a function in class $\mathcal{S}$. Since $f(f^{-1}(w))=w$ and using \eqref{eq-3.1} we obtain
\begin{align}\label{eq-3.2}
	\begin{cases}
		A_2= -a_2, \\ A_3=-a_3 +2a^2_{2}, \\ A_4=- a_4 +5a_2 a_3 -5a^3_{2}, \\ A_5=- a_5+6a_4 a_2- 21a_3 a^2_{2} +3a^2_{3} +14a^4_{2}.
	\end{cases}
\end{align}
The notation of the logarithmic coefficient of inverse of $f$ has been studied by Ponnusamy \textit{et al.} \cite{Ponnusamy-Sharma-Wirths-RM-2018}. As with $f$, the logarithmic inverse coefficients $\Gamma_n:=\Gamma_n(F)$, $n\in\mathbb{N}$, of $F$ are defined by the equation
\begin{align}\label{eq-3.3}
	F_{f^{-1}}(w):=\log\left(\frac{f^{-1}(w)}{w}\right)=2\sum_{n=1}^{\infty} \Gamma_n(F) w^n \;\;\;\; \mbox{for} \;\;|w|<1/4.
\end{align}
In $2018$, Ponnusamy \textit{et al.} \cite{Ponnusamy-Sharma-Wirths-RM-2018} obtained the sharp bound for the logarithmic inverse coefficients for functions in the class $\mathcal{S}$. In fact, Ponnusamy \emph{et al.}  \cite{Ponnusamy-Sharma-Wirths-RM-2018} established that for $f\in\mathcal{S}$ 
\begin{align*}
	|\Gamma_n(F)|\leq\frac{1}{2n}\binom{2n}{n}
\end{align*}
and showed that the equality holds only for the Koebe function or  its rotations. By differentiating \eqref{eq-3.3} together with \eqref{eq-3.1}, using \eqref{eq-3.2} and then equating coefficients, we obtain
\begin{align}\label{eq-3.4}
	\begin{cases}
		\Gamma_1=-\dfrac{1}{2}a_2, \vspace{1.5mm}\\ \Gamma_2=-\dfrac{1}{2}\left(a_3 -\dfrac{3}{2}a^2_{2}\right), \vspace{1.5mm} \\ \Gamma_3=-\dfrac{1}{2}\left(a_4 -4a_2 a_3 +\dfrac{10}{3}a^3_{2}\right).
	\end{cases}
\end{align}
In \cite{Kowalczyk-Lecko-BAMS-2022}, Kowalczyk and Lecko  proposed a Hankel determinant whose elements are the logarithmic coefficients of $f\in\mathcal{S}$, realizing the extensive use of these coefficients. Inspired by these ideas, S\"ummer \textit{et al.} in \cite{Sumer-Lecko-BMMSS-2023} started the investigation of the Hankel determinants $H_{q,n}(F_{f^{-1}}/2)$, wherein the entries of Hankel matrices are logarithmic coefficient of $f^{-1}$ with $f\in\mathcal{S}$. It follows that
\begin{align}\label{eq-3.5}
	H_{2,1}(F_{f^{-1}}/2):=\Gamma_{1}\Gamma_{3} -\Gamma^2_{2}=\frac{1}{48} \left(13a^4_2 -12a^2_2 a_3 - 12 a^2_3 + 12 a_2 a_4\right).
\end{align}
It is now appropriate to remark that $H_{2,1}(F_{f^{-1}}/2)$ is invariant under rotation, since for $f_{\theta}(z):=e^{-i\theta} f(e^{i\theta}z)$, $\theta\in\mathbb{R}$ when $f\in\mathcal{S}$, we have
\begin{align*}
	H_{2,1}(F_{f^{-1}_{\theta}}/2)=\frac{e^{4i\theta}}{48}\left(13a^4_2 -12a^2_2 a_3 - 12 a^2_3 + 12 a_2 a_4\right)=e^{4i\theta}H_{2,1} (F_{f^{-1}_{\theta}}/2).
\end{align*}
\subsection{Logarithmic inverse coefficients for the Class $\mathcal{BT_\mathfrak{B}}$:}We obtain the following result where, we obtain the sharp bound of the logarithmic coefficients of inverse functions in the class $ \mathcal{BT_\mathfrak{B}}$.
\begin{theorem}
Let $f(z)=z+a_2z^2+a_3z^3+\cdots\in\mathcal{BT_\mathfrak{B}}$ and $ \Gamma_{1}, \Gamma_{2} $ are given by \eqref{eq-3.4}. Then we have 
\begin{align*}
	|\Gamma_1|\leq \frac{1}{8}\;\;\;\;\mbox{and}\;\;\;\;|\Gamma_2|\leq \frac{1}{12}.
\end{align*}
The inequalities are sharp for the functions $f_1$ and $ f_2 $, defined in \eqref{eq-2.8} and \eqref{eq-2.9}, respectively.
\end{theorem}
\begin{proof}
\noindent{\bf Sharp bounds of $\Gamma_{1}$:} From \eqref{eq-2.16} and \eqref{eq-3.4}, we have
\begin{align}\label{eq-3.6}
	|\Gamma_1|=\;\vline-\frac{1}{2}a_2\;\vline=\vline-\frac{1}{16}c_1\vline \leq\frac{1}{8}.
\end{align}
The function $f_1$, which is defined in \eqref{eq-2.8} gives the sharpness of the inequality \eqref{eq-3.6}. \vspace{2mm}
	
\noindent{\bf Sharp bounds of $\Gamma_{2}$:} By using \eqref{eq-2.16} and \eqref{eq-3.4}, we have
\begin{align*}
	|\Gamma_2|&=\;\vline-\dfrac{1}{2}\left(a_3 -\dfrac{3}{2}a^2_{2} \right)\vline\\&=\; \vline-\dfrac{1}{2}\left(\dfrac{1}{12}c_2 - \dfrac{5}{96}c_1^2 -\dfrac{3}{128}c^2_{1} \right)\vline \\&=\frac{1}{24}\;\vline\; c_2 -\frac{29}{32}c^2_{1} \;\vline.
\end{align*}
Therefore, by Lemma C, we obtain
\begin{align}\label{eq-3.7}
	|\Gamma_2|\leq \frac{1}{12}.
\end{align}
The function $f_2$, which is defined in \eqref{eq-2.9} gives the sharpness of the inequality \eqref{eq-3.7}. This completes the proof.
\end{proof}
\subsection{\textbf{Sharp bound of $H_{2,1}(F_{f^{-1}}/2)$ for the Class $\mathcal{BT_\mathfrak{B}}$}:} We obtain a result finding the sharp bound of the second Hankel determinant $ H_{2,1}(F_{f^{-1}}/2) $ with logarithmic coefficients for functions in the class $ \mathcal{BT_\mathfrak{B}}$.
\begin{theorem}
Let $f\in\mathcal{BT_\mathfrak{B}}$ and has the series representation $f(z)=z+a_2z^2+a_3z^3+\cdots$, and $\Gamma_{1}, \Gamma_{2}$, $ \Gamma_{3}$ are given by \eqref{eq-3.4}. Then we have 
\begin{align*}
	|H_{2,1}(F_{f^{-1}}/2)|:=|\gamma_1\gamma_3-\gamma_2^2|\leq \frac{1}{144}.
\end{align*}
The inequality is sharp for the function $f_2$, defined in \eqref{eq-2.9}.
\end{theorem}
\begin{proof}
Let $f\in\mathcal{BT_\mathfrak{B}}$. Then there exists an analytic function $\omega\in\mathcal{H}$ with $\omega(0)=0$ and $|\omega(z)|<1$ for $z\in\mathbb{D}$, such that 
\begin{align}\label{eq-3.8}
	f^{\prime}(z)=\mathfrak{B}(\omega(z))=\sqrt{1+\tanh(\omega(z))},\;\; z\in\mathbb{D}.
\end{align}
Let $p\in\mathcal{P}$. Then, it may be written in terms of the Schwarz function $\omega$ by
\begin{align*}
	p(z)=\frac{1+\omega(z)}{1-\omega(z)}=1+c_1z+c_2z^2+\cdots,\; z\in\mathbb{D}.
\end{align*} 
Hence, it is evident that
\begin{align}\label{eq-3.9}
	\omega(z)&=\nonumber\frac{p(z)-1}{p(z)+1}\\&\nonumber=\frac{c_1}{2}z+\frac{1}{2}\left(c_2-\frac{c_1^2}{2}\right)z^2+\frac{1}{2}\left(c_3-c_1c_2+\frac{c_1^3}{4}\right)z^3\\&\quad+\frac{1}{2}\left(c_4-c_1c_3+\frac{3c_1^2c_2}{4}-\frac{c_2^2}{2}-\frac{c_1^4}{8}\right)z^4+\cdots
\end{align}
in $\mathbb{D}$. Then $p$ is analytic in $\mathbb{D}$ with $p(0)=1$ and has positive real part in $\mathbb{D}$. In view of \eqref{eq-3.8} together with $\mathcal{B}(\omega(z))$, a tedious computation shows that 
\begin{align}\label{eq-3.10}
	\nonumber\sqrt{1+\tanh(\omega(z))}&=1+\frac{1}{4}c_1z+\left( \frac{1}{4}c_2 -\frac{5}{32}c_1^2\right)z^2+ \left(\frac{1}{4} c_3 -\frac{5}{16}c_1c_2 +\frac{31}{384}c^3_1\right)z^3\\&\quad + \left(\frac{1}{4}c_4 -\frac{1}{4}c_1c_3 -\frac{5}{32}c_2^2 + \frac{7}{32}c^2_1 c_2 -\frac{223}{6144}c_1^4\right)z^4+\cdots
\end{align}
and 
\begin{align}\label{eq-3.11}
	f^{\prime}(z)=1+2a_2z+3a_3 z^2 +4a_4 z^3+ 5a_5 z^4+\cdots.
\end{align}
Thus, using \eqref{eq-3.10} and \eqref{eq-3.11}, we compute from \eqref{eq-3.8} that
\begin{align}\label{eq-3.12}
	\begin{cases}
		a_2=\dfrac{1}{8}c_1,\vspace{2mm}\\
		a_3=\dfrac{1}{12}c_2 - \dfrac{5}{96}c_1^2,\vspace{2mm}\\
		a_4=\dfrac{1}{16}c_3 - \dfrac{5}{64}c_1c_2 + \dfrac{31}{1536}c_1^3. 
	\end{cases}
\end{align}	
A simple computation by using \eqref{eq-3.5} and \eqref{eq-3.12}, shows
that
\begin{align}\label{eq-3.13}
	H_{2,1}(F_{f^{-1}}/2)&\nonumber=\frac{1}{48} \left(13a^4_2 -12a^2_2 a_3 - 12 a^2_3 + 12 a_2 a_4\right)\\&=\frac{1}{589824} \left(131c_1^4 - 352c_1^2c_2 - 1024c_2^2 + 1152c_1c_3\right).
\end{align}
By Lemma A and \eqref{eq-3.13}, we obtain
\begin{align}\label{eq-3.14}
	H_{2,1}(F_{f^{-1}}/2)\nonumber&=\frac{1}{36864} \bigg(-13\tau_1^4 -112\tau_1^2\tau_2\left(1-\tau_1^2\right)-32\tau_2^2(8+\tau_1^2)(1-\tau_1^2)\\&\quad+ 288\tau_1\tau_3 \left(1-\tau_1^2\right) \left(1-|\tau_2|^2\right)\bigg).
\end{align}
	
We now explore three possible cases involving $\tau_1$. \vspace{2mm}
	
\noindent{\bf Case-I.} Let $\tau_1=1$. Then, from \eqref{eq-3.14} we see that 
\begin{align*}
	|H_{2,1}(F_{f^{-1}}/2)|=\frac{13}{36864}\approx 0.000352.
\end{align*}
\noindent{\bf Case-II.} Let $\tau_1=0$. Then, from \eqref{eq-3.14} we get 
\begin{align*}  
	|H_{2,1}(F_{f^{-1}}/2)|=\bigg|\frac{1}{36864}\left(-256\tau_2^2\right)\bigg|\leq\frac{1}{144}\approx 0.0069444.
\end{align*}
\noindent{\bf Case-III.} Let $\tau_1\in (0, 1)$. Applying triangle inequality in \eqref{eq-3.14} and using the fact that $|\tau_3|\leq 1$, we obtain
\begin{align}\label{eq-3.15}
	|H_{2,1}(F_{f^{-1}}/2)|&\nonumber\leq\frac{1}{36864}\bigg(\bigg|-13\tau_1^4 -112\tau_1^2\tau_2\left(1-\tau_1^2\right)-32\tau_2^2(8 +\tau_1^2) (1-\tau_1^2)\bigg|\\&\nonumber\quad+288\tau_1(1- \tau_1^2)(1-|\tau_2|^2)\bigg)\\&=\frac{1}{128}\tau_1(1-\tau_1^2)\left(\vline\; A+B\tau_2+C\tau_2^2\;\vline +1-|\tau_2|^2\right)\nonumber\\& :=\frac{1}{128}\tau_1(1-\tau_1^2)Y(A, B, C),
\end{align}
where
\begin{align*}
	A=\frac{13\tau_1^3}{288(1-\tau_1^2)},\;\; B=\frac{7\tau_1}{18},\;\;\mbox{and}\;\; C=\frac{(8+\tau_1^2)}{9\tau_1}.
\end{align*}
We note that $AC>0$. Hence, we can apply case (i) of Lemma B and discuss the following cases.\\
	
A simple computation shows that
\begin{align*}
	|B|-2(1-|C|)=\frac{7\tau_1}{18}-2\left(1-\frac{8+\tau_1^2}{9\tau_1}\right)=\frac{32-36\tau_1+11\tau_1^2}{18\tau_1}>0
\end{align*}
for all $\tau_1\in (0, 1)$. \textit{i.e.,} $|B|>2(1-|C|)$. Thus from Lemma B, we see that 
\begin{align*}
	Y(A, B, C)=|A|+|B|+|C|=\frac{256- 112\tau_1^2 -131\tau_1^4} {288\tau_1\left(1-\tau_1^2\right)}.
\end{align*}
In view of the inequality \eqref{eq-3.15} it follows that
\begin{align}\label{eq-3.16}
	\nonumber|H_{2,1}(F_{f^{-1}}/2)|&=\frac{1}{128}\tau_1(1-\tau_1^2)\left(|A|+|B|+|C|\right)\\&\nonumber=\frac{1}{36864}\left(256- 112\tau_1^2 -131\tau_1^4\right)\\&=\frac{1}{36864}\Psi_1(\tau_1),
\end{align}
\begin{figure}[!htb]
	\begin{center}
		\includegraphics[width=0.55\linewidth]{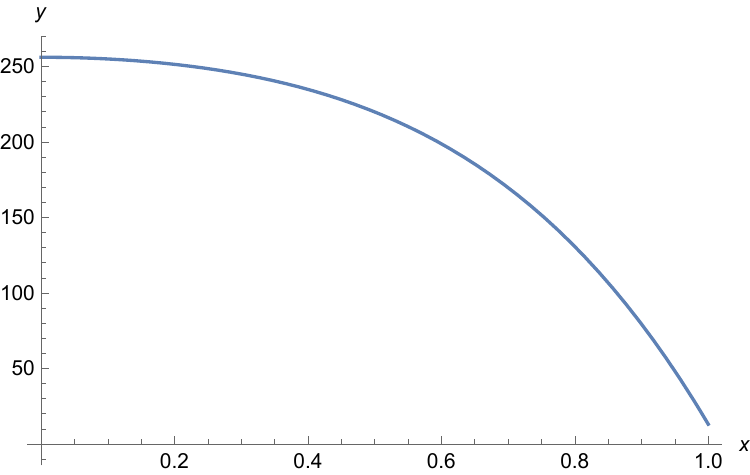}
	\end{center}
	\caption{The image $\Psi_1(t)$ for $0\leq t\leq 1$.}
\end{figure}
where $\Psi_1(t)=256-112t^2-131t^4$ for $t\in [0, 1]$. A simple computation shows that $\Psi^{\prime}_1(t)=-224t -524t^3<0$ for all $t\in (0,1]$ which shows that $\Psi_1$ is a decreasing function on $[0, 1]$. Hence, the maximum of $\Psi_1(t)$ is attained at $t=0$, and the maximum value is $256$. Hence, from \eqref{eq-3.16}, we see that
\begin{align}\label{eq-3.17}
	|H_{2,1}(F_{f^{-1}}/2)|\leq \frac{1}{144}.
\end{align}
By summarizing Cases I, II, and III, we obtain the desired inequality of the result. The function $f_2$, which is defined in \eqref{eq-2.9} gives the sharpness of the inequality \eqref{eq-3.17}. This completes the proof.
\end{proof}
\section{\bf Generalized Zalcman conjecture for the Class $\mathcal{BT_\mathfrak{B}}$:} In $1999$, Ma (see \cite{Ma-JMAA-1999}) generalized the  Zalcman conjecture as: for $f\in\mathcal{S}$, $|a_{n} a_{m} -a_{m+n-1}|\leq (n-1)(m-1)$; $n,m\in \mathbb{N}\setminus\{1\}$. In recent years there has been a great deal of attention devoted to finding sharp bounds of the Zalcman functional $J_{2,3}:=a_2 a_3 -a_4$ for several class of functions (see \cite{Lecko-Sim-RM-2019, Ravichandran- Verma-JMAA-2017, Krushkal-GMJ-2010}). Now we compute the sharp bounds of the generalized Zalcman functional $J_{2,3}:=a_2 a_3 -a_4$ for the  class $\mathcal{BT_\mathfrak{B}}$ being a special case of the generalized Zalcman functional $J_{n,m}:=a_n a_m -a_{n+m-1}$, $n,m\in\mathbb{N}\setminus\{1\}$, which was investigated by Ma in \cite{Ma-JMAA-1999} for $f\in\mathcal{S}$. We have derived the following result concerning the sharp bound for the Zalcman function $J_{2,3}$ in the class $\mathcal{BT_\mathfrak{B}}$.
\begin{theorem}
Let $f(z)=z+a_2z^2+a_3z^3+\cdots\in\mathcal{BT_\mathfrak{B}}$. Then we have
\begin{align}\label{Eq-4.1}
	|a_2 a_3 -a_4|\leq \frac{1}{8}.
\end{align}
The inequality \eqref{Eq-4.1} is sharp for the function $f_3$, which is defined in \eqref{eq-2.10}.
\end{theorem}
\begin{proof}
In view of \eqref{eq-3.12}, we obtain
\begin{align*}
	|a_2 a_3 -a_4|&=\;\vline\; \dfrac{1}{8}c_1 \left(\dfrac{1}{12}c_2 - \dfrac{5}{96}c_1^2\right) - \left(\dfrac{1}{16}c_3 - \dfrac{5}{64}c_1c_2 + \dfrac{31}{1536}c_1^3\right)\;\vline \\&=\;\vline -\frac{41}{1536} c^3_1 +\frac{17}{192}c_1 c_2 -\frac{1}{16}c_3\;\vline \\&=\frac{1}{1536}\;\vline -41 c^3_1 +136 c_1 c_2 -96c_3\;\vline\\&=\frac{1}{16}\;\vline\; c_3 -\frac{17}{12} c_1 c_2 +\frac{41}{96} c^3_1\;\vline\\& =\frac{1}{16}\;\vline\; c_3 -2Bc_1 c_2 +Dc^3_1\;\vline,
\end{align*}
where
\begin{align*}
	B=\frac{17}{24} \;\;\;\;\;\mbox{and}\;\;\;\; D=\frac{41}{96}.
\end{align*}
It is easy to see that $0\leq B\leq 1$ and the inequality $B(2B-1)\leq D\leq B$ implies that $\frac{85}{288}< \frac{41}{96}<\frac{17}{24}$, is true. Therefore, using Lemma D, we obtain
\begin{align}\label{eq-4.1}
	|a_2 a_3 -a_4|\leq \frac{1}{8}.
\end{align}
The function $f_3$, defined in \eqref{eq-2.10} gives the sharpness of the inequality \eqref{eq-4.1}. This completes the proof.
\end{proof}

\section{\bf Moduli differences of logarithmic coefficients for $\mathcal{BT_\mathfrak{B}}$} In $ 1985 $, de Branges \cite{Branges-AM-1985} solved the famous Bieberbach conjecture by showing that for any function $ f\in\mathcal{S} $ of the form \eqref{eq-1.1}, the inequality $ |a_n|\leq n $ holds for all $ n\geq 2 $, with equality attained by the Koebe function $ k(z):=z/(1-z)^2 $ or its rotations. This naturally led to the question of whether the inequality  $ ||a_{n+1}|-|a_n||\leq 1 $ holds for $ f\in\mathcal{S} $ when $ n\geq 2 $. This problem was first studied by Goluzin in \cite{Goluzin-1946} initially investigated this problem in an attempt to solve the Bieberbach conjecture. Later, in $ 1963 $, Hayman \cite{Hayman-1963} established that for all $ f\in\mathcal{S} $, there exists an absolute constant $ A\geq 1 $ such that $ ||a_{n+1}|-|a_n||\leq A $. The current best known estimate as of now is $ A=3.61 $, due to Grinspan \cite{Grinspan-1976}. On the other hand, for the class $ \mathcal{S} $, the sharp bound is known only for $ n=2 $ (see \cite[Theorem 3.11]{Duren-1983-NY}), namely
\begin{align*}
	-1\leq |a_3|-|a_2|\leq 1.029...
\end{align*}
Similarly, for functions $ f\in\mathcal{S}^* $, Pommerenke \cite{Ch. Pommerenke-1971} has conjectured that $ ||a_{n+1}|-|a_n||\leq 1 $ which was subsequently proven by Leung \cite{Leung-BLMS-1978} in $ 1978 $. For convex functions, Li and Sugawa \cite{Li-Sugawa-CMFT-2017} investigated the sharp bound of $ |a_{n+1}|-|a_n| $ for $ n\geq 2 $, and establish the sharp bounds for $ n=2, 3 $. \vspace{2mm}

The inverse functions are studied by several authors in different perspective (see, for instance, \cite{Sim-Thomas-BAMS-2022, Sim-Thomas-S-2020}). Recently, Sim and Thomas  \cite{Sim-Thomas-BAMS-2022} obtained sharp upper and lower bounds on the difference of the moduli of successive inverse coefficients
for the subclasses of univalent functions. Inspired by the prior research, including the recent article \cite{Allu-Shaji-JMAA-2025}, this paper focuses on determining sharp lower and upper bounds of $ |\gamma_2|-|\gamma_1| $ and $ |\Gamma_2|-|\Gamma_1| $ for functions in the class $ \mathcal{BT_\mathfrak{B}} $. Our approach involves proving Theorem \ref{Th-7.1} and Theorem \ref{Th-7.2} with the aid of Lemma F, which plays a crucial role. We state Lemma F as follows.

\begin{LemF}\cite{Sim-Thomas-S-2020}	
Let $B_1$, $B_2$ and $B_3$ be numbers such that $B_1>0$, $B_2\in\mathbb{C}$ and $B_3\in\mathbb{R}$. Let $p\in\mathcal{P}$ of the form \eqref{eq-2.4}. Define $\Psi_{+}(c_1,c_2)$ and $\Psi_{-}(c_1,c_2)$ by
\begin{align*}
	\Psi_{+}(c_1,c_2)=|B_2 c^2_1 +B_3 c_2| -|B_1 c_1|,
\end{align*}
and
\begin{align*}
	\Psi_{-}(c_1,c_2)=-\Psi_{+}(c_1,c_2).
\end{align*}
Then
\begin{align}\label{eq-5.1}
	\Psi_{+}(c_1,c_2)\leq
	\begin{cases}
		|4B_2 +2B_3|-2B_1, \;\;\;\;\mbox{if}\;\;|2B_2 +B_3|\geq |B_3|+ B_1,\vspace{2mm} \\ 2|B_3|,\hspace{2.8cm}\;\mbox{otherwise},
	\end{cases}
\end{align}
and
\begin{align}\label{eq-5.2}
	\Psi_{-}(c_1,c_2)\leq
	\begin{cases}
		2B_1 -B_4, \hspace{2.8cm}\mbox{if}\;\; B_1\geq B_4 +2|B_3|, \vspace{2mm} \\ 2B_1 \sqrt{\dfrac{2|B_3|}{B_4+2|B_3|}}, \hspace{1.3cm}\;\mbox{if}\;\;B^2_1\leq 2|B_3|(B_4 +2|B_3|), \vspace{2mm} \\ 2|B_3| +\dfrac{B^2_1}{B_4+2|B_3|}, \hspace{1cm}\;\mbox{otherwise},
	\end{cases}
\end{align}
where $B_4=|4B_2 +2B_3|$. All inequalities in \eqref{eq-5.1} and \eqref{eq-5.2} are sharp.
\end{LemF}
We have established the following result on the sharp inequality for the moduli differences of logarithmic coefficients in the class $\mathcal{BT_\mathfrak{B}}$.
\begin{theorem}\label{Th-7.1}
Let $f\in\mathcal{BT_\mathfrak{B}}$ and has the series representation $f(z)=z+a_2z^2+a_3z^3+\cdots$, and $ \gamma_{1}, \gamma_{2} $ are given by \eqref{eq-2.2}. Then we have 
\begin{align*}
	-\frac{1}{2\sqrt{23}}\leq|\gamma_2|-|\gamma_1|\leq \frac{1}{12}.
\end{align*}
Both inequalities are sharp.
\end{theorem}
\begin{proof}
In view of \eqref{eq-2.2} and \eqref{eq-2.16}, we see that
\begin{align}\label{eq-5.3}
	\nonumber|\gamma_2|-|\gamma_1|&=\;\vline\; \frac{1}{2} \left(a_{3} -\frac{1}{2}a^2_{2}\right)\;\vline - \;\vline\; \frac{1}{2} a_{2}\;\vline\\&\nonumber=\;\vline\; -\frac{23}{768}c^2_1 +\frac{1}{24}c_2\;\vline - \;\vline\; \frac{1}{16} c_1\;\vline\\&=\Psi_{+}(c_1,c_2),
\end{align}
where
\begin{align*}
	B_1=\frac{1}{16},\;\; B_2=-\frac{23}{768} \;\;\mbox{and}\;\; B_3= \frac{1}{24}.
\end{align*}
\noindent{\bf Estimate of the upper bound:} For the upper bound, we see that $|2B_2 +B_3|=\frac{7}{384}$ and $|B_3|+ B_1=\frac{5}{48}$. It follows that $|2B_2 +B_3|\not\geq |B_3|+ B_1$. Hence, applying Lemma F, we obtain
\begin{align*}
	\Psi_{+}(c_1,c_2)\leq 2|B_3|=\frac{1}{12}.
\end{align*}
As a result, applying \eqref{eq-5.3}, we obtain
\begin{align}\label{eq-5.4}
	|\gamma_2|-|\gamma_1|\leq \frac{1}{12}.
\end{align}
The function $f_2$, which is defined in \eqref{eq-2.9} gives the sharpness of the inequality \eqref{eq-5.4}.\vspace{2mm}
	
\noindent{\bf Estimate of the lower bound:} For the lower bound, we see that $B_4=|4B_2 +2B_3|=\frac{7}{192}$, $B_4 +2|B_3|=\frac{23}{192}$. It follows that $B_1\not\geq B_4 +2|B_3| $. Again, we have $2|B_3|(B_4 +2|B_3|)=\frac{23}{2304}$, so the condition
$B^2_1\leq 2|B_3|(B_4 +2|B_3|)$ is true. Thus, by Lemma F, we have
\begin{align*}
	\Psi_{-}(c_1,c_2)\leq 2B_1 \sqrt{\dfrac{2|B_3|}{B_4+2|B_3|}} =\frac{1}{2\sqrt{23}}.
\end{align*}
Clearly, we observe that
\begin{align*}
	\Psi_{+}(c_1,c_2)=-\Psi_{-}(c_1,c_2)\geq -\frac{1}{2\sqrt{23}}.
\end{align*}
Hence, from \eqref{eq-5.3}, we conclude
\begin{align}\label{eq-5.5}
	|\gamma_2|-|\gamma_1|\geq -\frac{1}{2\sqrt{23}}.
\end{align}
The inequality \eqref{eq-5.5} is sharp for the function $f\in\mathcal{A}$ given by \eqref{eq-2.12} with
\begin{align*}
	p(z)=\frac{1+\frac{8}{\sqrt{23}}z+z^2}{1-z^2},
\end{align*}
which completes the proof.
\end{proof}

We have established the following result on the sharp inequality for the moduli differences of logarithmic inverse coefficients in the class $\mathcal{BT_\mathfrak{B}}$.
\begin{theorem}\label{Th-7.2}
Let $f\in\mathcal{BT_\mathfrak{B}}$ and has the series representation $f(z)=z+a_2z^2+a_3z^3+\cdots$, and $ \Gamma_{1}, \Gamma_{2} $ are given by \eqref{eq-3.4}. Then we have 
\begin{align*}
	-\frac{1}{2\sqrt{29}}\leq|\Gamma_2|-|\Gamma_1|\leq \frac{1}{12}.
\end{align*}
Both inequalities are sharp.
\end{theorem}
\begin{proof}
In view of \eqref{eq-2.2} and \eqref{eq-2.16}, we see that
\begin{align}\label{eq-7.3}
	\nonumber|\gamma_2|-|\gamma_1|&=\;\vline\; -\dfrac{1}{2} \left(a_3 -\dfrac{3}{2}a^2_{2}\right)\;\vline - \;\vline\; -\frac{1}{2} a_{2}\;\vline\\&\nonumber=\;\vline\;\dfrac{1}{2} \left(a_3 -\dfrac{3}{2}a^2_{2}\right)\;\vline - \;\vline\; \frac{1}{2} a_{2}\;\vline\\&\nonumber=\;\vline\; -\frac{29}{768}c^2_1 +\frac{1}{24}c_2\;\vline - \;\vline\; \frac{1}{16} c_1\;\vline\\&=\Psi_{+}(c_1,c_2),
\end{align}
where
\begin{align*}
	B_1=\frac{1}{16},\;\; B_2=-\frac{29}{768} \;\;\mbox{and}\;\; B_3= \frac{1}{24}.
\end{align*}
\noindent{\bf Estimate of the upper bound:} For the upper bound, we see that $|2B_2 +B_3|=\frac{13}{384}$ and $|B_3|+ B_1=\frac{5}{48}$. It is easy to check that $|2B_2 +B_3|\not\geq |B_3|+ B_1$. Hence, in view of the Lemma F, we obtain
\begin{align*}
	\Psi_{+}(c_1,c_2)\leq 2|B_3|=\frac{1}{12}.
\end{align*}
Thus, it follows from \eqref{eq-7.3} that
\begin{align}\label{eq-7.4}
	|\Gamma_2|-|\Gamma_1|\leq \frac{1}{12}.
\end{align}
The function $f_2$, which is defined in \eqref{eq-2.9} gives the sharpness of the inequality \eqref{eq-7.4}.\vspace{2mm}
	
\noindent{\bf Estimate of the lower bound:} For the lower bound, we see that $B_4=|4B_2 +2B_3|=\frac{13}{192}$, $B_4 +2|B_3|=\frac{29}{192}$. It follows that $B_1\not\geq B_4 +2|B_3| $. Again, we have $2|B_3|(B_4 +2|B_3|)=\frac{29}{2304}$, so the condition
$B^2_1\leq 2|B_3|(B_4 +2|B_3|)$ is true. Thus, by Lemma F, we have
\begin{align*}
	\Psi_{-}(c_1,c_2)\leq 2B_1 \sqrt{\dfrac{2|B_3|}{B_4+2|B_3|}} =\frac{1}{2\sqrt{29}}.
\end{align*}
A simple computation leads to 
\begin{align*}
	\Psi_{+}(c_1,c_2)=-\Psi_{-}(c_1,c_2)\geq -\frac{1}{2\sqrt{29}}.
\end{align*}
Consequently, from \eqref{eq-5.3}, we obtain
\begin{align}\label{eq-7.5}
	|\Gamma_2|-|\Gamma_1|\geq -\frac{1}{2\sqrt{29}}.
\end{align}
The inequality \eqref{eq-7.5} is sharp for the function $f\in\mathcal{A}$ given by \eqref{eq-2.12} with
\begin{align*}
	p(z)=\frac{1+\frac{8}{\sqrt{29}}z+z^2}{1-z^2},
\end{align*}
which completes the proof.
\end{proof}

\section{\bf Conclusion}
We have obtained sharp bounds for the logarithmic coefficients, as well as for the second-order Hankel determinants involving logarithmic coefficients, in the class of functions characterized by bounded turning, which is associated with the intriguing bean-shaped domain. Notably, all the bounds we established are sharp. Furthermore, various important properties of these functions have been analyzed, including estimates related to the generalized Zalcman and moduli differences of logarithmic coefficients. \vspace{2mm}\\

\noindent{\bf Acknowledgment:} The authors would like to thank the referee(s) for their helpful suggestions and comments for the improvement of the exposition of the paper. The first author is supported by CSIR-SRF (File No: 09/0096(12546)/2021-EMR-I, dated: 08/10/2024), Govt. of India, New Delhi and the second author supported by SERB File No. SUR/2022/002244, Govt. of India.\\

\noindent\textbf{Compliance of Ethical Standards:}\\

\noindent\textbf{Conflict of interest.} The authors declare that there is no conflict  of interest regarding the publication of this paper.\vspace{1.5mm}

\noindent\textbf{Data availability statement.}  Data sharing is not applicable to this article as no datasets were generated or analyzed during the current study.

\end{document}